\documentclass[11pt,twoside]{article}
\usepackage{amsmath, amsthm, amssymb}
\newtheorem{definition}{Definition}
\newtheorem{theorem}{Theorem}
\newtheorem{lemma}{Lemma}
\newtheorem{remark}{Remark}
\newtheorem{example}{Example}

\newtheorem{proposition}{Proposition}
\makeatletter \oddsidemargin0.45in \evensidemargin \oddsidemargin
\begin{document}
\thispagestyle{empty} \setcounter{page}{1}



\begin{center}
{\Large\bf  On fractional derivatives with exponential kernel and their discrete versions }

\vskip.20in
\vskip.20in
 Thabet Abdeljawad$^{a}$, Dumitru Baleanu$^{b,c}$ \\[2mm]
{\footnotesize $^{a}$Department of Mathematics and Physical Sciences,
Prince Sultan
University\\ P. O. Box 66833,  11586 Riyadh, Saudi Arabia\\
 $^{b}$Department of Mathematics and Computer Science, \c{C}ankaya University\\ 06530 Ankara, Turkey\\
 $^{c}$ Institute of Space Sciences, Magurele-Bucharest, Romania}
\end{center}

\vskip.2in

{\footnotesize \noindent {\bf Abstract}.In this manuscript we define the right fractional derivative and its corresponding right fractional integral with exponential kernel. Then, we provide the integration by parts formula and  we use $Q-$operator to confirm our results. The related  Euler-Lagrange equations were obtained and one  example is analysed. Moreover, we formulate and discus the discrete counterparts of our results.

}

\vskip.1in

\section{Introduction} \label{s:1}

The techniques of the fractional calculus are applied successfully in many branches of science and engineering \cite{podlubny, Samko,Kilbas,Magin,dumitru}. The power law effects are described accurately within the fractional calculus approach.
However there are several complex phenomena which do not obey this law, thus we need to find alternatives, e.g. the nonlocal kernels without singularity in order to overcome this issue.
On the other side the discrete version of any fractional derivative is one of the interesting topics nowadays \cite{Th Caputo}-\cite{dualCaputo}. Some of the  applications of the discrete fractional Caputo derivative can be seen in \cite{Wu1, Wu2}.

 For two real numbers $a<b,~~a\equiv~b~(mod~1)$, we denote $\mathbb{N}_a=\{a,a+1,..\}$ and $~_{b}\mathbb{N}=\{b,b-1,...\}$. Also, $\nabla f(t)=f(t)-f(t-1).$

\indent
The $Q-$operator action, $(Qf)(t)=f(a+b-t)$, was used in \cite{dualCaputo, dualR} to relate left and right fractional sums and differences. It was shown that
\begin{itemize}
  \item $(\nabla_a^{-\alpha} Qf)(t)=Q~_{b}\nabla^{-\alpha} f(t)$
  \item $(\nabla_a^{\alpha} Qf)(t)=Q~_{b}\nabla^\alpha f(t)$
  \item $(~^{C}\nabla_a^{\alpha} Qf)(t)=Q~_{b}~^{C}\nabla^\alpha f(t)$
\end{itemize}
The mixing of delta and nabla operators in defining right fractional differences plays an important role in obtaining the above dual identities.

In this article, we will use the action of the discrete version of $Q-$operator to define and confirm our definitions of fractional differences with discrete exponential function kernels.

We recall some facts about discrete transform within nabla. For more general theory of Laplace transform on time scales we refer to \cite{Martin, Adv}.
The (nabla) exponential function on the time scale $\mathbb{Z}$ is given by
\begin{equation}\label{nabla exponential}
  \widehat{e}(\lambda, t-t_0)= (\frac{1}{1-\lambda})^{t-t_0},
\end{equation}
while the (delta) exponential function by
\begin{equation}\label{delta exponential}
 \widetilde{ e}(\lambda, t-t_0)=  (1+\lambda)^{t-t_0}
\end{equation}
In this manuscript, we follow the nabla discrete analysis.
\begin{definition}\label{dl}
The nabla discrete Laplace transform for a function $f$ defined on $\mathbb{N}_0$ is defined by

\begin{equation}\label{dle}
   \mathcal{N} f(z)=\sum_{t=1}^\infty (1-z)^{t-1}f(t),
\end{equation}

More generally for a function $f$  defined on $\mathbb{N}_a$ by
\begin{equation}\label{gdle}
   \mathcal{N}_a f(z)=\sum_{t=a+1}^\infty (1-z)^{t-1}f(t).
   \end{equation}

\end{definition}

\begin{definition} \label{conv}(See \cite{Thsemi} also)
Let $s \in \mathbb{R}$, $0<\alpha <1$ and $f,g:\mathbb{N}_a\rightarrow \mathbb{R}$ be a functions. The nabla discrete convolution of $f$ with $g$ is defined by

\begin{equation}\label{dconv}
   ( f\ast g)(t)=\sum_{s=a+1}^t g(t-\rho(s)) f(s).
\end{equation}
\end{definition}

\begin{proposition} \label{convprop}
For any $\alpha \in \mathbb{R}\setminus\{...,-2,-1,0\}$ , $s \in \mathbb{R}$ and $f,g$ defined on $\mathbb{N}_a$ we have
\begin{equation}\label{conv1}
    (\mathcal{N}_a (f \ast g))(z)=  (\mathcal{N}_af)(z) (\mathcal{N} g)(z).
\end{equation}
\end{proposition}

\begin{proof}

\begin{eqnarray}
\nonumber
 (\mathcal{N}_a (f \ast g))(z) &=& \sum_{t=a+1}^\infty (1-z)^{t-1} \sum_{a+1}^t f(s) g(t-\rho(s)) \\ \nonumber
   &=& \sum_{s=a+1}^\infty \sum_{t=s}^\infty (1-z)^{t-1} \sum_{a+1}^t f(s) g(t-\rho(s)) \\ \nonumber
   &=&  \sum_{s=a+1}^\infty \sum_{r=1}^\infty (1-z)^{r-1}(1-z)^{s-1}f(s)g(r)\\ \nonumber
   &=&  (\mathcal{N}_af)(z) (\mathcal{N} g)(z),
\end{eqnarray}
where the change of variable $r=t-\rho(s)$ was used.

\end{proof}

For the delta Laplace convolution theory we refer to (\cite{Martin}, Section 3.10).

\begin{lemma}\label{lap of nabla}\cite{Thsemi}
Let $f$ be defined on $\mathbb{N}_0$. Then

\begin{equation}\label{lap of 1}
    (\mathcal{N} \nabla(f(t))(z)=z (\mathcal{N}f)(z)-f(0).
\end{equation}
\end{lemma}

We can generalize Lemma \ref{lap of nabla} as follows:

\begin{lemma}\label{general lap of nabla}
Let $f$ be defined on $\mathbb{N}_a$. Then

\begin{equation}\label{lap of 1}
    (\mathcal{N}_a \nabla(f(t))(z)=z (\mathcal{N}_af)(z)-(1-z)^a f(a).
\end{equation}

\end{lemma}

It is well-known that the (delta or nabla )Laplace  transform of the constat function $1$ on any time scale is $\frac{1}{z}$ and the Lapalace transform  of the exponential function is $\frac{1}{z-\lambda}$.

\indent

\indent

\section{Left and right fractional derivatives with exponential nonsingular kernel and their correspondent integral operators}
In this section we will define the discrete fractional derivatives with nabla exponential kernels. From the nabla time scale calculus (see \cite{Martin}, page 55), we recall that the nabla exponential function on the $\mathbb{Z}$ (or $\mathbb{N}_a$) time scale is given by $\widehat{e}_\lambda(t,a)=(\frac{1}{1-\lambda})^{t-a},~~~~t \in \mathbb{N}_a$. According to \cite{FCaputo}, if $f \in H^1(a,b),~~a<b,~~\alpha \in [0,1]$, then, the new (left) Caputo fractional derivative is:
\begin{equation}\label{nCd}
 ( ~^{CFC}~_{a}D^\alpha f)(t)=\frac{M(\alpha)}{1-\alpha}\int_a^t f^\prime(s) exp(\lambda (t-s))ds,~~ \lambda= \frac{-\alpha}{1-\alpha}.
\end{equation}
The right new (Caputo) fractional derivative can be defined by
\begin{equation}\label{nCRd}
 ( ~^{CFC}D_b^\alpha f)(t)=-\frac{M(\alpha)}{1-\alpha}\int_t^b f^\prime(s) exp(\lambda (s-t))ds,~~ \lambda= \frac{-\alpha}{1-\alpha}.
\end{equation}

and the left new Riemann fractional derivative by
\begin{equation}
 ( ~^{CFR}~_{a}D^\alpha f)(t)=\frac{M(\alpha)}{1-\alpha}\frac{d}{dt}\int_a^t f(s) exp(\lambda (t-s))ds,~~ \lambda= \frac{-\alpha}{1-\alpha}.
\end{equation}

and the right new Riemann fractional derivative by
\begin{equation}
 ( ~^{CFR}D_b^\alpha f)(t)=-\frac{M(\alpha)}{1-\alpha}\frac{d}{dt}\int_t^b f(s) exp(\lambda (s-t))ds,~~ \lambda= \frac{-\alpha}{1-\alpha}.
\end{equation}
It can be easily shown that $( ~^{CFR}~_{a}D^\alpha Q f)(t)=Q( ~^{CFR}D_b^\alpha  f)(t)$ and $( ~^{CFC}~_{a}D^\alpha Q f)(t)=Q( ~^{CFC}D_b^\alpha  f)(t)$.

Consider the equation $(~^{CFR}~_{a}D^\alpha f)(t)=u(t)$. If we apply Laplace transform starting at $a$ to both sides and use the convolution theorem , we conclude that
$$F(s)=\frac{1-\alpha}{B(\alpha)} U(s)-\frac{\alpha}{B(\alpha)} \frac{U(s)}{s},$$
where $F(s)=L_a\{f(t)\}(s)$ and $U(s)=L_a\{u(t)\}(s)$. Then, apply the inverse Laplace to reach at

$$f(t)=\frac{1-\alpha}{B(\alpha)}u(t)+\frac{\alpha}{B(\alpha)}\int_a^tu(s)ds.$$ Therefore, we can define the correspondent left fractional integral
of $(~^{CF}~_{a}D^\alpha)$ by
\begin{equation}\label{lfrac}
  (~^{CF}~_{a}I^\alpha u)(t)=\frac{1-\alpha}{B(\alpha)}u(t)+\frac{\alpha}{B(\alpha)}\int_a^tu(s)ds.
\end{equation}

We define the right fractional integral
\begin{equation}\label{rfrac}
  (~^{CF}I_b^\alpha u)(t)=\frac{1-\alpha}{B(\alpha)}u(t)+\frac{\alpha}{B(\alpha)}\int_t^b u(s)ds,
\end{equation}

 so that $(~^{CF}~_{a}I^\alpha Qu)(t)=Q(~^{CF}I_b^\alpha u)(t)$.
Furthermore, by the action of the $Q$-operator we can easily show that $(~^{CF}I_b^\alpha ~^{CF}D_b^\alpha u )(t)=u(t)$.

\indent

Conversely, consider the equation $$(~^{CF}~_{a}I^\alpha u)(t)=\frac{1-\alpha}{B(\alpha)}u(t)+\frac{\alpha}{B(\alpha)}\int_a^tu(s)ds=f(t).$$
Apply the Laplace transform to see that $$\frac{1-\alpha}{B(\alpha)}U(s)=\frac{\alpha}{sB(\alpha)} F(s).$$

From which it follows that

\begin{equation}\label{xx1}
  U(s)=\frac{B(\alpha)}{1-\alpha} \frac{sF(s)}{s-\lambda}.
\end{equation}
The right hand side of (\ref{xx1}) is nothing but the Laplace of $(~^{CFR}~_{a}D^\alpha f)(t)$. Hence, $$(~^{CF}~_{a}I^\alpha~^{CFR}~_{a}D^\alpha f)(t)=f(t).$$

Similarly or by the action of the $Q-$operator we can show that $$(~^{CF}I_b^\alpha~^{CF}D_b^\alpha f)(t)=f(t).$$

\begin{remark} (The action of the  $Q-$operator)The $Q$-operator acts regularly between left and right new fractional differences as follow:

\begin{itemize}
  \item $(Q~^{CFR}~_{a}D^\alpha f)(t)= (~^{CFR}D^\alpha_b Qf)(t)$
  \item $(Q~^{CFC}~_{a}D^\alpha f)(t)= (~^{CFC}D^\alpha_b Qf)(t)$
\end{itemize}

\end{remark}

\begin{proposition} (The relation between Riemann and Caputo type fractional derivatives with exponential kernels) \label{relation between C and R}

 \begin{itemize}
   \item $(~^{CFC}~_{a}D^\alpha f)(t)=(~^{CFR}~_{a}D^\alpha f)(t)-\frac{B(\alpha)}{1-\alpha} f(a)e^{\lambda(t-a)}$.
   \item $(~^{CFC}D_b^\alpha f)(t)=(~^{CFR}D_b^\alpha f)(t)-\frac{B(\alpha)}{1-\alpha} f(b)e^{\lambda(b-t)}$.
 \end{itemize}
 \end{proposition}
\begin{proof}
From the relations
\begin{equation}\label{rel1}
  L_a\{(~^{CFR}~_{a}D^\alpha f)(t)\}(s)= \frac{B(\alpha)}{1-\alpha} \frac{sF(s)}{s-\lambda},
\end{equation}
and
\begin{equation}\label{rel2}
  L_a\{(~^{CFC}~_{a}D^\alpha f)(t)\}(s)= \frac{B(\alpha)}{1-\alpha} \frac{sF(s)}{s-\lambda}-\frac{B(\alpha)}{1-\alpha}f(a)e^{-as}\frac{1}{s-\lambda},
\end{equation}
we conclude that
\begin{equation}\label{rel3}
   L_a\{(~^{CFC}~_{a}D^\alpha f)(t)\}(s)= L_a\{(~^{CFR}~_{a}D^\alpha f)(t)\}(s)-\frac{B(\alpha)}{1-\alpha}f(a)e^{-as}\frac{1}{s-\lambda}.
\end{equation}
Applying the inverse Laplace to (\ref{rel3}) and making use of the fact that $L_a\{f(t-a)\}(s)=e^{-as}L\{f(t)\}(s)$ we reach our conclusion in the first part. The second part can be proved by the first part and the action of the $Q-$operator.
\end{proof}
The  following result  is very useful tool to solve fractional dynamical systems within Caputo fractional derivative with exponential  kernals.
\begin{proposition}
For $0< \alpha< 1$, we have

\begin{eqnarray}\label{fintegral of Caputo}
 \nonumber
  ( ~^{AB}~_{a}I^\alpha ~^{ABC}~_{a}D^\alpha f)(x) &=& f(x)-f(a)e^{\lambda(x-a)}-\frac{\alpha}{1-\alpha}f(a)\int_a^x e^{\lambda(x-s)}ds \\  \nonumber
   &=& f(x)-f(a).
\end{eqnarray}
 Similarly,

 \begin{equation}\label{rfintegral of Caputo}
   ( ~^{AB}I_b^\alpha ~^{ABC}D_b^\alpha f)(x)=f(x)-f(b)
 \end{equation}
\end{proposition}
\section{Integration by parts for fractional derivatives with exponential kernels}

Before we present an integration by part formula for the new proposed fractional derivatives and integrals we introduce the following function spaces:
For $p\geq1$ and $\alpha >0$, we define

\begin{equation}\label{nn1}
  (~^{CF}~_{a}I^\alpha (L_p)=\{f: f=~^{CF}~_{a}I^\alpha \varphi, ~~\varphi \in L_p(a,b)\}.
\end{equation}
and
\begin{equation}\label{nn2}
  (~^{CF}I_b^\alpha (L_p)=\{f: f=~^{CF}I_b^\alpha \phi, ~~\phi \in L_p(a,b)\}.
\end{equation}

\indent

Above it was  shown that the left fractional   operator $~^{CFR}~_{a}D^\alpha$ and its associate fractional integral $~^{CF}~_{a}I^\alpha$ satisfy $(~^{CFR}~_{a}D^\alpha ~^{CF}~_{a}I^\alpha f)(t)=f(t)$ and  that  $(~^{CFR}D_b^\alpha ~^{CF}I_b^\alpha f)(t)=f(t)$ . Also it was shown that $(~^{CF}~_{a}I^\alpha ~^{CFR}~_{a}D^\alpha  f)(t)=f(t)$ and $(~^{CF}I_b^\alpha ~^{CFR}D_b^\alpha  f)(t)=f(t)$ and hence the function spaces $(~^{CF}~_{a}I^\alpha (L_p)$ and $(~^{CF}I_b^\alpha (L_p)$ are nonempty.

\begin{theorem} (Integration by parts)\label{Integration by parts}
Let $\alpha >0$, $p\geq 1,~q \geq 1$, and $\frac{1}{p}+\frac{1}{q}\leq 1+\alpha$ ($p\neq1$  and $q\neq1$ in the case $\frac{1}{p}+\frac{1}{q}=1+\alpha$ ). Then
\begin{itemize}
  \item If  $\varphi(x) \in L_p(a,b) $ and $\psi(x) \in L_q(a,b)$ , then
   \begin{eqnarray}
            \nonumber
             \int_a^b \varphi(x) (~^{FC}~_{a}I^\alpha\psi)(x)dx &=& \frac{1-\alpha}{B(\alpha)}\int_a^b \psi(x)\varphi(x)dx+ \frac{\alpha}{B(\alpha)}\int_a^b \psi(x) (\int_x^b \psi(t)dt )dx \\
              &=& \int_a^b \psi(x) (~^{CF}I_b^\alpha\varphi(x)dx
           \end{eqnarray}

      and similarly,
      \begin{eqnarray}
      \nonumber
        \int_a^b \varphi(x) (~^{CF}I_b^\alpha\psi)(x)dx&=& \frac{1-\alpha}{B(\alpha)}\int_a^b \psi(x)\varphi(x)dx+ \frac{\alpha}{B(\alpha)}\int_a^b \psi(x)(\int_a^x \varphi(t) dt)dx \\
         &=&  \int_a^b \psi(x) (~^{CF}~_{a}I^\alpha\varphi)(x)dx
      \end{eqnarray}

  \item If $f(x) \in ~^{CF}I_b^\alpha (L_p) $ and $g(x) \in ~^{CF}~_{a}I^\alpha (L_q)$, then  $$\int_a^b f(x) (~^{CFR}~_{a}D^\alpha g)(x)dx=\int_a^b (~^{CFR}D_b^\alpha f)(x) g(x)dx$$
\end{itemize}
\end{theorem}
\begin{proof}
\begin{itemize}
  \item From the definition and the identity
  \begin{equation}\label{ident1}
    \int_a^b \psi(x) (\int_a^x\varphi(t)dt)dx=\int_a^b \varphi(t) (\int_t^b \psi(x)dx)dt
  \end{equation}
  we have
  \begin{eqnarray}
         \nonumber
          \int_a^b \varphi(x) (~^{CF}~_{a}I^\alpha\psi)(x)dx &=& \int_a^b \varphi(x) [\frac{1-\alpha}{B(\alpha)}\psi(x)+\frac{\alpha}{B(\alpha)}  \int_a^x\psi(t)dt  ]dx  \\ \nonumber
           &=&\frac{1-\alpha}{B(\alpha)} \int_a^b \varphi(x) \psi(x) dx \\ \nonumber
           &+& \frac{\alpha}{B(\alpha)}\int_a^b \psi(x)(\int_x^b \varphi(t)dt) dx \\
           &=& \int_a^b \psi(x) (~^{CF}I_b^\alpha\varphi(x)dx.
  \end{eqnarray}
  The other case follows similarly by the definition of the right fractional integral in ( \ref{rfrac}) and the identity (\ref{ident1}).
  \item From definition and the first part we have
  \begin{eqnarray}
   \nonumber
    \int_a^b f(x) (~^{CFR}~_{a}D^\alpha g)(x)dx &=& \int_a^b ( ~^{CF}I_b^\alpha \phi)(x).(~ ~^{CFR}~_{a}D^\alpha  \circ~^{CF} ~_{a}I^\alpha \varphi)(x)dx\\ \nonumber
     &=& \int_a^b ( ~^{CF}I_b^\alpha \phi)(x). \varphi(x)dx \\ \nonumber
     &=& \int_a^b  \phi(x) (~^{CF}~_{a}I^\alpha\varphi)(x)dx \\ \nonumber
      &=& \int_a^b (~^{CFR}D_b^\alpha f)(x) g(x)dx.
     \end{eqnarray}

\end{itemize}

\end{proof}
Before proving a by part formula for Caputo fractional derivatives, we use the following two notations:
\begin{itemize}
  \item  The (left) exponential   integral operator
\begin{equation}\label{lIO}
( \textbf{ e}_{\lambda,a^+}\varphi )(x)=\int_a^x e^{\lambda(t-a)}\varphi(t) dt,~~x>a.
\end{equation}
  \item The (right) exponential   integral operator

\begin{equation}\label{rIO}
( \textbf{ e}_{ \lambda,b^-}\varphi )(x)=\int_x^b e^{\lambda(b-t)}\varphi(t) dt,~~x<b.
\end{equation}
\end{itemize}

\begin{theorem} (Integration by parts for the Caputo derivatives with exponential kernel)
 \label{C by parts}
\begin{itemize}
  \item $\int_a^b (~^{CFC}~_{a}D^\alpha f)(t)g(t)= \int_a^b f(t) (~^{CFR}D_b^\alpha g)(t)+ \frac{B(\alpha)}{1-\alpha} f(t) \textbf{ e}_{ \frac{-\alpha}{1-\alpha},b^-}g )(t)|_a^b$.
  \item $\int_a^b (~^{CFC}D_b^\alpha f)(t)g(t)= \int_a^b f(t) (~^{CFR}~_{a}D^\alpha g)(t)- \frac{B(\alpha)}{1-\alpha} f(t) \textbf{ e}_{ \frac{-\alpha}{1-\alpha},a^+}g )(t)|_a^b$.
  \end{itemize}

\end{theorem}

 \begin{proof}
  The proof of the first part follows by Theorem \ref{Integration by parts} and the first part of Proposition \ref{relation between C and R} and the proof of the second part follows by Theorem \ref{Integration by parts} and  the second part of Proposition \ref{relation between C and R}.
\end{proof}

\section{Fractional Euler-Lagrange Equations}
We prove Euler-Lagrange equations for a Lagrangian containing the left new Caputo derivative with exponential kernel.

\begin{theorem}\label{A1}
Let $0<\alpha \leq 1$ be non-integer, $a,b \in \mathbb{R},~~a<b$,  Assume that the functional $J:C^2[a,b]\rightarrow  \mathbb{R}$ of the form
 $$J(f)=\int_{a}^{b} L(t,f(t),~^{CFC}~_{a}D^\alpha f(t) )dt$$
 has a local extremum in $S=\{y \in C^2[a,b]: ~~y(a)=A, y(b)=B\}$ at some $f \in S$, where
 $L:[a,b]\times \mathbb{R}\times \mathbb{R}\rightarrow \mathbb{R}$. Then,
\begin{equation}\label{E1}
[L_1(s) + ~^{ABR}D_b^\alpha L_2(s)]  =0,~\texttt{for all}~ s \in [0,b],
\end{equation}
where $L_1(s)= \frac{\partial L}{\partial f}(s)$ and $L_2(s)=\frac{\partial L}{\partial ~^{CFC}~_{0}D^\alpha f}(s)$.
\end{theorem}
\begin{proof}
Without loss of generality, assume that $J$ has local maximum in $S$ at $f$. Hence, there exists an $\epsilon>0$ such that $J(\widehat{f})-J(f)\leq 0$ for all $\widehat{f}\in S$ with $\|\widehat{f}-f\|=\sup_{t \in \mathbb{N}_a \cap ~_{b}\mathbb{N}} |\widehat{f}(t)-f(t)|< \epsilon$. For any $\widehat{f} \in S$ there is an $\eta \in H=\{y \in C^2[a,b], ~~y(a)=y(b)=0\}$ such that $\widehat{f}=f+\epsilon \eta$. Then, the $\epsilon-$Taylor's theorem implies that
$$L(t,f,\widehat{f})=$$
$$L(t,f+\epsilon \eta,~^{CFC}~_{a}D^\alpha f+\epsilon ~^{CFC}~_{a}D^\alpha \eta)=L(t,f,~^{CFC}~_{a}D^\alpha f)+ \epsilon [\eta L_1+~^{CFC}~_{a}D^\alpha \eta L_2]+O(\epsilon^2).$$ Then,

\begin{eqnarray}\nonumber
  J(\widehat{f})-J(f) &=& \int_a^bL(t,\widehat{f}(t),~^{CFC}~_{a}D^\alpha \widehat{f}(t))-\int_0^b L(t,f(t),~^{CFC}~_{a}D^\alpha f(t)) \\
  &=& \epsilon \int_a^b[\eta(t) L_1(t)+ (~^{CFC}~_{a}D^\alpha \eta)(t) L_2(t)]+ O(\epsilon^2).
   \end{eqnarray}
Let the quantity $\delta J(\eta,y)=\int_a^b[\eta(t) L_1(t)+ (~^{CFC}~_{a}D^\alpha \eta)(t) L_2(t)]dt$ denote the first variation of $J$.

Evidently, if $\eta \in H$ then $-\eta \in H$, and $\delta J(\eta,y)=-\delta J(-\eta,y)$. For $\epsilon$ small, the sign of $J(\widehat{f})-J(f)$ is determined by the sign of first variation, unless $\delta J(\eta,y)=0$ for all $\eta \in H$. To make the parameter $\eta$ free, we use the integration by part formula in Theorem \ref{C by parts}, to reach
$$\delta J(\eta,y)=\int_0^b \eta (s)[L_1(s) + ~^{CFR}D_b^\alpha L_2(s)]+\eta(t)\frac{B(\alpha)}{1-\alpha} (\textbf{ e}_{ \frac{-\alpha}{1-\alpha},b^-}L_2 )(t)|_a^b =0,$$ for all $\eta \in H$, and hence the result follows by the fundamental Lemma of calculus of variation.
\end{proof}
The term $(\textbf{ e}_{ \frac{-\alpha}{1-\alpha},b^-}L_2 )(t)|_0^b =0$ above is called the natural boundary condition.

\indent
Similarly, if we allow the Lagrangian to depend on the right Caputo fractional derivative, we can state:
\begin{theorem}\label{A2}
Let $0<\alpha \leq 1$ be non-integer, $a,b \in \mathbb{R},~~a<b$,  Assume that the functional $J:C^2[a,b]\rightarrow  \mathbb{R}$ of the form
 $$J(f)=\int_{a}^{b} L(t,f(t),~^{CFC}D_b^\alpha f(t) )dt$$
 has a local extremum in $S=\{y \in C^2[a,b]: ~~y(a)=A, y(b)=B\}$ at some $f \in S$, where
 $L:[a,b]\times \mathbb{R}\times \mathbb{R}\rightarrow \mathbb{R}$. Then,
\begin{equation}\label{E1}
[L_1(s) + ~^{CFR}~_{a}D^\alpha L_2(s)]  =0,~\texttt{for all}~ s \in [a,b],
\end{equation}
where $L_1(s)= \frac{\partial L}{\partial f}(s)$ and $L_2(s)=\frac{\partial L}{\partial ~^{CFC}D_b^\alpha f}(s)$.
\end{theorem}
\begin{proof}
The proof is similar to  Theorem \ref{A1} by applying the second integration by parts in Proposition \ref{C by parts} to get the natural boundary condition of the form $(\textbf{ e}_{ \frac{-\alpha}{1-\alpha},0^+}L_2 )(t)|_a^b =0$.
\end{proof}
\begin{example}
In order to exemplify our results we study an example of physical interest under Theorem \ref{A1}. Namely, let us consider the following fractional  action,

  $J(y)=\int_a^b[\frac{1}{2}( ~^{CFC}~_{a}D^\alpha y(t))^2-V(y(t))],$ where $0<\alpha <1$ and with $y(a),~~y(b)$ are assigned or with the natural boundary condition $(\textbf{ e}^1_{\alpha,1, \frac{-\alpha}{1-\alpha},b^-}~^{CFC}~_{0}D^\alpha y(t) )(t)|_a^b =0 $.  Then,  the Euler-Lagrange equation by applying Theorem \ref{A1} is
       $$(~^{CFR}D_b^\alpha ~o ~~^{CFC}~_{0}D^\alpha y)(s)-\frac{dV}{dy}(s)=0~\texttt{for all}~ s \in [0,b].$$
       Clearly, if we let $\alpha\rightarrow 1$, then the Euler-Lagrange equation for the above functional is reduced to $y^{\prime\prime}=0$ when the potential functions is zero. Its solution is then linear, which agrees with the classical case.
       For equations in the classical fractional case with composition of left and right fractional derivatives with the action of the $Q-$operator we refer to \cite{TDFdelay}.
\end{example}

\section{Discrete versions of fractional derivatives with exponential kernels}

Using the time scale notation,  the nabla discrete exponential kernel can be expressed as $\widehat{e}_\lambda(t,\rho(s))=(\frac{1}{1-\lambda})^{t-\rho(s)}=(1-\alpha)^{t-\rho(s)}$ and  $\widetilde{e}_\lambda(t,\sigma(s))=(1+\lambda)^{t-\sigma(s)}=(\frac{1-2\alpha}{1-\alpha})^{t-\sigma(s)}$ , where
$\lambda= \frac{-\alpha}{1-\alpha}$. Hence, we can propose the following discrete versions:

\begin{definition} \label{nabla staff}
For $\alpha \in (0,1)$ and $f$ defined on $\mathbb{N}_a$, or $~_{b}\mathbb{N}$ in right case, we define
\begin{itemize}
  \item The left (nabla)  new Caputo fractional difference by

   \begin{eqnarray}
    \nonumber
     (~^{CFC}~_{a}\nabla^\alpha f)(t) &=&  \frac{B(\alpha)}{1-\alpha}\sum_{s=a+1}^t(\nabla_s f)(s) (1-\alpha)^{t-\rho(s)}\\
      &=&  B(\alpha)\sum_{s=a+1}^t(\nabla_s f)(s) (1-\alpha)^{t-s}
   \end{eqnarray}

  \item The right (nabla) new Caputo fractional difference by
   \begin{eqnarray}
    \nonumber
     (~^{CFC}\nabla^\alpha_b f)(t) &=& \frac{B(\alpha)}{1-\alpha}\sum_{s=t}^{b-1}(-\Delta_s f)(s) (1-\alpha)^{s-\rho(t)} \\
      &=& B(\alpha)\sum_{s=t}^{b-1}(-\Delta_s f)(s) (1-\alpha)^{s-t}
   \end{eqnarray}

    \item The left (nabla) new Riemann fractional difference by
   \begin{eqnarray}
    \nonumber
     (~^{CFR}~_{a}\nabla^\alpha f)(t) &=&  \frac{B(\alpha)}{1-\alpha}\nabla_t\sum_{s=a+1}^t f(s) (1-\alpha)^{t-\rho(s)}\\
      &=&  B(\alpha)\nabla_t \sum_{s=a+1}^t f(s) (1-\alpha)^{t-s}.
   \end{eqnarray}
  \item The right (nabla) new Riemann fractional difference by
  \begin{eqnarray}
    \nonumber
     (~^{CFR}\nabla^\alpha_b f)(t) &=& \frac{B(\alpha)}{1-\alpha}(-\Delta_t)\sum_{s=t}^{b-1} f(s) (1-\alpha)^{s-\rho(t)} \\
      &=& B(\alpha) (-\Delta_t)\sum_{s=t}^{b-1}f(s) (1-\alpha)^{s-t}
   \end{eqnarray}

\end{itemize}
\end{definition}

\begin{remark}In the limiting case $\alpha\rightarrow 0$ and  $\alpha\rightarrow 1$, we remark the following

\begin{itemize}
  \item  $$(~^{CFC}~_{a}\nabla^\alpha f)(t)\rightarrow f(t)-f(a)~~ \texttt{as } ~~\alpha\rightarrow 0,$$
  and
        $$(~^{CFC}~_{a}\nabla^\alpha f)(t)\rightarrow \nabla  f(t))~~ \texttt{as } ~~\alpha\rightarrow 1.$$
  \item $$ (~^{CFC}\nabla^\alpha_b f)(t)\rightarrow f(t)-f(b)~~ \texttt{as } ~~\alpha\rightarrow 0,$$
  and

  $$ (~^{CFC}\nabla^\alpha_b f)(t)\rightarrow -\Delta f(t)~~ \texttt{as } ~~\alpha\rightarrow 1.$$
  \item $$(~^{CFR}~_{a}\nabla^\alpha f)(t)\rightarrow f(t)  ~~ \texttt{as } ~~\alpha\rightarrow 0,$$

  and
  $$(~^{CFR}~_{a}\nabla^\alpha f)(t)\rightarrow \nabla f(t)  ~~ \texttt{as } ~~\alpha\rightarrow 1.$$

  \item $$(~^{CFR}\nabla^\alpha_b f)(t)\rightarrow f(t)~~ \texttt{as } ~~\alpha\rightarrow 0,$$

  and
  $$(~^{CFR}\nabla^\alpha_b f)(t)\rightarrow -\Delta f(t)~~ \texttt{as } ~~\alpha\rightarrow 1,$$

\end{itemize}

\end{remark}

\begin{remark} (The action of the discrete $Q-$operator)The $Q$-operator acts regularly between left and right new fractional differences as follows:

\begin{itemize}
  \item $(Q~^{CFR}~_{a}\nabla^\alpha f)(t)= (~^{CFR}\nabla^\alpha_b Qf)(t)$
  \item $(Q~^{CFC}~_{a}\nabla^\alpha f)(t)= (~^{CFC}\nabla^\alpha_b Qf)(t)$
\end{itemize}

\end{remark}

Consider the equation $(~^{CFR}~_{a}\nabla^\alpha f)(t)=u(t)$. If we apply $(\mathcal{N}_a$ the discrete nabla Laplace transform starting at $a$ to both sides and use the convolution theorem , we conclude that
$$F(s)=\frac{1-\alpha}{B(\alpha)}[U(s)-\frac{\lambda}{s}U(s)]  =\frac{1-\alpha}{B(\alpha)} U(s)-\frac{\alpha}{B(\alpha)} \frac{U(s)}{s},$$
where $F(s)=(\mathcal{N}_a\{f(t)\}(s)$ and $U(s)=(\mathcal{N}_a\{u(t)\}(s)$. Then, apply the inverse Laplace to reach at

$$f(t)=\frac{1-\alpha}{B(\alpha)}u(t)+\frac{\alpha}{B(\alpha)}\sum_{s=a+1}^tu(s)ds.$$ Therefore, we can define the correspondent left fractional integral
of $(~^{CFR}~_{a}D^\alpha)$ by
\begin{equation}\label{dlfrac}
  (~^{CF}~_{a}\nabla^{-\alpha} u)(t)=\frac{1-\alpha}{B(\alpha)}u(t)+\frac{\alpha}{B(\alpha)}\sum_{s=a+1}^tu(s)ds.
\end{equation}

We define the right fractional integral
\begin{equation}\label{drfrac}
  (~^{CF}\nabla^{-\alpha}_b u)(t)=\frac{1-\alpha}{B(\alpha)}u(t)+\frac{\alpha}{B(\alpha)}\sum_{s=t}^{b-1}u(s)ds,
\end{equation}

 so that $(~^{CF}~_{a}\nabla^{-\alpha}Qu)(t)=Q(~^{CF}\nabla^{-\alpha}_b u)(t)$.
Furthermore, by the action of the $Q$-operator we can easily show that $(~^{CF}\nabla^{-\alpha}_b~^{CF}\nabla_b^\alpha u )(t)=u(t)$.

\indent

Conversely, consider the equation $$(~^{CF}~_{a}\nabla^{-\alpha} u)(t)=\frac{1-\alpha}{B(\alpha)}u(t)+\frac{\alpha}{B(\alpha)}\sum_{a+1}^tu(s)=f(t).$$
Apply the nabla discrete  Laplace transform to see that $$\frac{1-\alpha}{B(\alpha)}U(s)=\frac{\alpha}{sB(\alpha)} F(s).$$

From which it follows that

\begin{equation}\label{dxx1}
  U(s)=\frac{B(\alpha)}{1-\alpha} \frac{sF(s)}{s-\lambda}.
\end{equation}
The right hand side of (\ref{dxx1}) is nothing but the Laplace of $(~^{CF}~_{a}\nabla^\alpha f)(t)$. Hence, $$(~^{CF}~_{a}\nabla^{-\alpha}~^{CF}~_{a}\nabla^\alpha f)(t)=f(t).$$

Similarly or by the action of the $Q-$operator we can show that $$(~^{CF}\nabla^{-\alpha}_b~^{CF}\nabla_b^\alpha f)(t)=f(t).$$

\begin{proposition} (The relation between Riemann and Caputo type fractional differences with exponential kernels) \label{drelation between C and R}

 \begin{itemize}
   \item $(~^{CFC}~_{a}\nabla^\alpha f)(t)=(~^{CFR}~_{a}\nabla^\alpha f)(t)-\frac{B(\alpha)}{1-\alpha} f(a)(1-\alpha)^{(t-a)}$.
   \item $(~^{CFC}\nabla_b^\alpha f)(t)=(~^{CFR}\nabla_b^\alpha f)(t)-\frac{B(\alpha)}{1-\alpha} f(b)(1-\alpha)^{(b-t)}$.
 \end{itemize}
 \end{proposition}
\begin{proof}
From the relations
\begin{equation}\label{drel1}
  \mathcal{N}_a\{(~^{CFR}~_{a}\nabla^\alpha f)(t)\}(s)= \frac{B(\alpha)}{1-\alpha} \frac{sF(s)}{s-\lambda},
\end{equation}
and
\begin{equation}\label{drel2}
  \mathcal{N}_a\{(~^{CFC}~_{a}D^\alpha f)(t)\}(s)= \frac{B(\alpha)}{1-\alpha} \frac{sF(s)}{s-\lambda}-\frac{B(\alpha)}{1-\alpha}f(a)(1-s)^a\frac{1}{s-\lambda},
\end{equation}
we conclude that
\begin{equation}\label{drel3}
   \mathcal{N}_a\{(~^{CFC}~_{a}D^\alpha f)(t)\}(s)= \mathcal{N}_a\{(~^{CFR}~_{a}D^\alpha f)(t)\}(s)-\frac{B(\alpha)}{1-\alpha}f(a)(1-s)^a\frac{1}{s-\lambda}.
\end{equation}
Applying the inverse Laplace to (\ref{drel3}) and making use of the fact that $\mathcal{N}_a\{f(t-a)\}(s)=(1-s)^a\mathcal{N}\{f(t)\}(s)$ we reach our conclusion in the first part. The second part can be proved by the first part and the action of the $Q-$operator.
\end{proof}

\section{Integration by parts for fractional differences with discrete  exponential kernels}

Above it was  shown that the left fractional   operator $~^{CFR}~_{a}\nabla^\alpha$ and its associate fractional integral $~^{CF}~_{a}\nabla^{-\alpha}$ satisfy $(~^{CFR}~_{a}\nabla^\alpha {CF}~_{a}\nabla^{-\alpha} f)(t)=f(t)$ and  that  $(~^{CFR}D_b^\alpha ~^{CF}\nabla^{-\alpha}_b f)(t)=f(t)$ . Also it was shown that $(~^{CF}~_{a}\nabla^{-\alpha} ~^{CFR}~_{a}\nabla^\alpha  f)(t)=f(t)$ and $(~^{CF}\nabla^{-\alpha}_b ~^{CFR}\nabla_b^\alpha  f)(t)=f(t)$ .

\begin{theorem} (Integration by parts for fractional differences)\label{discrete Integration by parts}
Assume $0<\alpha <1$ and the used functions are defined on $\mathbb{N}_a \cap ~_{b}\mathbb{N},~~a\equiv ~b ~(mod ~1)$. Then
\begin{itemize}
  \item For functions  $\varphi(t) $ and $\psi(t)$ , we have
   \begin{eqnarray}
            \nonumber
             \sum_{t=a+1}^{b-1} \varphi(t) (~^{CF}~_{a}\nabla^{-\alpha}\psi)(t) &=& \frac{1-\alpha}{B(\alpha)}\sum_{t=a+1}^{b-1} \psi(t)\varphi(t)+ \frac{\alpha}{B(\alpha)} \sum_{t=a+1}^{b-1} \psi(t) \sum_{s=t}^{b-1} \varphi(s)  \\
              &=& \sum_{t=a+1}^{b-1} \psi(t) ~^{CF}\nabla^{-\alpha}_b \varphi(t)
           \end{eqnarray}

      and similarly,
      \begin{eqnarray}
      \nonumber
        \sum_{t=a+1}^{b-1} \varphi(t) (~^{CF}\nabla^{-\alpha}_b\psi)(t)&=& \frac{1-\alpha}{B(\alpha)}\sum_{t=a+1}^{b-1} \psi(t)\varphi(t)+ \frac{\alpha}{B(\alpha)}\sum_{t=a+1}^{b-1} \psi(t)(\sum_{s=a+1}^t \varphi(s))  \\
         &=&  \sum_{t=a+1}^{b-1} \psi(t) (~^{CF}~_{a}\nabla^{-\alpha} \varphi)(t).
      \end{eqnarray}

  \item For functions $f(t)$ and $g(t)$ defined on $\mathbb{N}_a$ we have,
   $$\sum_{t=a+1}^{b-1} f(t) (~^{CFR}~_{a}\nabla^\alpha g)(t)=\sum_{t=a+1}^{b-1} (~^{CFR}\nabla_b^\alpha f)(t) g(t).$$
\end{itemize}
\end{theorem}
\begin{proof}
\begin{itemize}
  \item From the definition and the identity
  \begin{equation}\label{dident1}
    \sum_{a+1}^{b-1} \psi(t) (\sum_{s=a+1}^t\varphi(s))=\sum_{a+1}^{b-1} \varphi(s) (\sum_{t=s}^{b-1} \psi(t)),
  \end{equation}
  the first part follows.
  The other case follows similarly by the definition of the right fractional difference in ( \ref{drfrac}) and the identity (\ref{dident1}).
  \item From the statement before the Theorem and the first part we have
  \begin{eqnarray}
   \nonumber
    \sum_{t=a+1}^{b-1} f(t) (~^{CFR}~_{a}\nabla^\alpha g)(t) &=&  \sum_{t=a+1}^{b-1} (~^{CF}\nabla_b^{-\alpha} ~^{CFR}\nabla_b^\alpha f)(t) (~^{CFR}~_{a}\nabla^\alpha g)(t)
    \\ \nonumber
     &=&  \sum_{t=a+1}^{b-1} ( ~^{CFR}\nabla_b^\alpha f)(t) ( ~^{CF}~_{a}\nabla^{-\alpha}  ~^{CFR}~_{a}\nabla^\alpha g)(t)\\ \nonumber
     &=& \sum_{t=a+1}^{b-1} ( ~^{CFR}\nabla_b^\alpha f)(t)  g(t) .
     \end{eqnarray}

\end{itemize}

\end{proof}
Before proving a by part formula for Caputo fractional derivatives, we use the following two notations:
\begin{itemize}
  \item  The (left) discrete exponential   integral operator
\begin{equation}\label{dlIO}
( \textbf{ E }_{\lambda,a^+} \varphi )(t)=\sum_{s=a+1}^t (\frac{1}{1-\lambda})^{(t-a)}\varphi(s),~~t \in \mathbb{N}_a.
\end{equation}
  \item The (right) discrete exponential   integral operator

\begin{equation}\label{drIO}
( \textbf{ E}_{ \lambda,b^-}\varphi )(x)=\sum_{t=s}^{b-1}(\frac{1}{1-\lambda})^ {(b-t)}\varphi(t) ,~~t \in ~_{b}\mathbb{N}.
\end{equation}
\end{itemize}

\begin{theorem} (Integration by parts for the Caputo derivatives with exponential kernel)
 \label{dC by parts}
\begin{itemize}
  \item $\sum_{t=a}^{b-1} (~^{CFC}~_{a-1}\nabla^\alpha f)(t)g(t)= \sum_{t=a}^{b-1}f(t) (~^{CFR}\nabla_{b-1}^\alpha g)(t)+ \frac{B(\alpha)}{1-\alpha} f^\rho(t) \textbf{ E}_{ \frac{-\alpha}{1-\alpha},b^-}g )(t)|_a^b$.
  \item $\sum_{t=a+1}^{b} (~^{CFC}\nabla_{b+1}^\alpha f)(t)g(t)= \sum_{t=a+1}^{b}  f(t) (~^{CFR}~_{a+1}\nabla^\alpha g)(t)- \frac{B(\alpha)}{1-\alpha} f^\sigma(t) \textbf{ E}_{ \frac{-\alpha}{1-\alpha},a^+}g )(t)|_a^b$.
  \end{itemize}

\end{theorem}

 \begin{proof}
  The proof of the first part follows by Theorem \ref{discrete Integration by parts} and the first part of Proposition \ref{drelation between C and R} and the proof of the second part follows by Theorem \ref{discrete Integration by parts} and  the second part of Proposition \ref{drelation between C and R}.
\end{proof}

\section{Discrete fractional Euler-Lagrange equations}

We prove Euler-Lagrange equations for a Lagrangian containing the left new discrete Caputo derivative with discrete exponential kernel.

\begin{theorem}\label{dA1}
Let $0<\alpha < 1$ be non-integer, $a,b \in \mathbb{R},~~a<b,~~~~a\equiv~ b~(mod~1)$.  Assume that the functional  of the form
 $$J(f)=\sum_{t=a}^{b-1} L(t,f^\rho(t),~^{CFC}~_{a-1}\nabla^\alpha f(t) )$$
 has a local extremum in $S=\{y:(\mathbb{N}_{a-1} \cap ~_{b-1}\mathbb{N})\rightarrow\mathbb{R}: ~~y(a-1)=A, y(b-1)=B\}$ at some $f \in S$, where
 $L:(\mathbb{N}_{a-1} \cap ~_{b-1}\mathbb{N})\times \mathbb{R}\times \mathbb{R}\rightarrow \mathbb{R}$. Then,
\begin{equation}\label{E1}
[L_1(s) + ~^{CFR}\nabla_{b-1}^\alpha L_2(s)]  =0,~\texttt{for all}~ s \in (\mathbb{N}_{a-1} \cap ~_{b-1}\mathbb{N}),
\end{equation}
where $L_1(s)= \frac{\partial L}{\partial f^\rho}(s)$ and $L_2(s)=\frac{\partial L}{\partial ~^{CFC}~_{a-1}\nabla^\alpha f}(s)$.
\end{theorem}
\begin{proof}
Without loss of generality, assume that $J$ has local maximum in $S$ at $f$. Hence, there exists an $\epsilon>0$ such that $J(\widehat{f})-J(f)\leq 0$ for all $\widehat{f}\in S$ with $\|\widehat{f}-f\|=\sup_{t \in \mathbb{N}_a \cap ~_{b}\mathbb{N}} |\widehat{f}(t)-f(t)|< \epsilon$. For any $\widehat{f} \in S$ there is an $\eta \in H=\{y :(\mathbb{N}_{a-1} \cap ~_{b-1}\mathbb{N})\rightarrow\mathbb{R}:, ~~y(a-1)=y(b-1)=0\}$ such that $\widehat{f}=f+\epsilon \eta$. Then, the $\epsilon-$Taylor's theorem and the assumption implies that the first variation  quantity $\delta J(\eta,y)=\sum_{t=a}^{b-1}[\eta^\rho(t) L_1(t)+ (~^{CFC}~_{a-1}\nabla^\alpha \eta)(t) L_2(t)]dt=0$

 for all $\eta \in H$. To make the parameter $\eta$ free, we use the first integration by part formula in Theorem  \ref{dC by parts}, to reach
$$\delta J(\eta,f)=\sum_{s=a}^{b-1} \eta^\rho (s)[L_1(s) + ~^{CFR}\nabla_{b-1}^\alpha L_2(s)]+\eta^\rho(t)\frac{B(\alpha)}{1-\alpha} (\textbf{ E}_{ \frac{-\alpha}{1-\alpha},b^-}L_2 )(t)|_a^b =0,$$ for all $\eta \in H$, and hence the result follows by the discrete fundamental Lemma of calculus of variation.
\end{proof}
The term $(\textbf{ E}_{ \frac{-\alpha}{1-\alpha},b^-}L_2 )(t)|_a^b =0$ above is called the natural boundary condition.

\indent
Similarly, if we allow the Lagrangian to depend on the discrete right Caputo fractional derivative, we can state:
\begin{theorem}\label{dA2}
Let $0<\alpha \leq 1$ be non-integer, $a,b \in \mathbb{R},~~a<b,~~~~a\equiv~ b~(mod~1)$.  Assume that the functional $J$ of the form
 $$J(f)=\sum_{a+1}^{b} L(t,f^\sigma(t),~^{CFC}\nabla_{b+1}^\alpha f(t) )$$
 has a local extremum in $S=\{y :(\mathbb{N}_{a+1} \cap ~_{b+1}\mathbb{N})\rightarrow\mathbb{R}: ~~y(a+1)=A, y(b+1)=B\}$ at some $f \in S$, where
 $L:(\mathbb{N}_{a+1} \cap ~_{b+1}\mathbb{N})\times \mathbb{R}\times \mathbb{R}\rightarrow \mathbb{R}$. Then,
\begin{equation}\label{E1}
[L_1(s) + ~^{CFR}~_{a+1}\nabla^\alpha L_2(s)]  =0,~\texttt{for all}~ s \in (\mathbb{N}_{a+1} \cap ~_{b+1}\mathbb{N}),
\end{equation}
where $L_1(s)= \frac{\partial L}{\partial f^\sigma}(s)$ and $L_2(s)=\frac{\partial L}{\partial ~^{CFC}\nabla_{b+1}^\alpha f}(s)$.
\end{theorem}
\begin{proof}
The proof is similar to  Theorem \ref{dA1} by applying the second integration by parts formula in Theorem \ref{dC by parts} to get the natural boundary condition of the form $(\textbf{ E}_{ \frac{-\alpha}{1-\alpha},a^+}L_2 )(t)|_a^b =0$.
\end{proof}
\begin{example}
Below we provide a discrete  example of physical interest under Theorem \ref{dA1}. We consider the following fractional discrete action,

  $J(y)=\sum_{t=a}^{b-1}[\frac{1}{2}( ~^{CFC}~_{a-1}\nabla^\alpha y(t))^2-V(y^\rho(t))],$ where $0<\alpha <1$ and with $y(a-1),~~y(b-1)$ are assigned or with the natural boundary condition $$(\textbf{ E}_{ \frac{-\alpha}{1-\alpha},b^-}~^{CFC}~_{a-1}\nabla^\alpha y(t) )(t)|_a^b =0. $$  Then,  the Euler-Lagrange equation by applying Theorem \ref{dA1} is
       $$(~^{CFR}\nabla_{b-1}^\alpha ~o ~~^{CFC}~_{a-1}\nabla^\alpha y)(s)-\frac{dV}{dy}(s)=0~\texttt{for all}~ s \in (\mathbb{N}_{a-1}\cap ~_{b-1}\mathbb{N}).$$
       For the sake of comparisons with the classical discrete fractional Euler-Lagrange equations within nabla we refer to \cite{ThELE}.
\end{example}

\section{Conclusions}

Caputo-Fabrizio derivative  was already applied successfully since it was recently suggested. Several applications of it were implemented  in some areas where the non-locality appears in real world phenomena.
However, the properties of this derivative should be deeply explored and applied to describe the real world phenomena.
In this manuscript we elaborated  the  discrete version of the Caputo-Fabrizio derivative. By proving the integration by parts we open the  gate for application in discrete variational principles and their applications in control theory. The reported discrete fractional Euler-Lagrange will play a crucial role in constructing the related Lagrangian and Hamiltonian dynamics.


\end{document}